\documentclass{amsart}

\usepackage[utf8]{inputenc}
\usepackage[T1]{fontenc}
\usepackage{amsmath,amsthm,amsopn,amstext,amscd,amsfonts,amssymb,mathrsfs,mathtools}
\usepackage{float}
\usepackage{tikz}
\usepackage[hidelinks]{hyperref}

\DeclareMathOperator{\diag}{diag}
\newtheorem{theorem}{\sc Theorem}[section]
\theoremstyle{remark}
\newtheorem{remark}[theorem]{\sc Remark}
\theoremstyle{plain}

\begin{document}

\title{Symmetric products of Laguerre zeros}
\author{K. Castillo}
\address{CMUC, Department of Mathematics, University of Coimbra,
3000-143 Coimbra, Portugal}
\email{kenier@mat.uc.pt}
\date{\today}

\begin{abstract}
We prove and substantially extend a conjecture of Gazeau, Josse-Michaux,
and Monceau concerning the extreme positive zeros of Hermite polynomials,
which arose from a finite-dimensional quantisation of the phase plane.
For the generalised Laguerre polynomial $L_m^{(\alpha)}$ of degree $m$
and parameter $\alpha>-1$, the product of its $j$th smallest and $j$th
largest zeros is strictly increasing with $m$ for each fixed positive
integer $j$, once $m\ge2j-1$. The classical
Hermite--Laguerre identities then prove the original conjecture for the
degree-$N$ Hermite polynomial $H_N$, in both parities and for every
$N\ge4$. As a further consequence, a one-parameter extension involving
a deformation of the corresponding Jacobi matrix is settled.
\end{abstract}

\maketitle

\section{Introduction\label{intro}}

In 2006, Gazeau, Josse-Michaux, and Monceau
\cite[Sections~4--5]{GazeauJosseMonceau2006} considered a
finite-dimensional quantisation of the phase plane built from the first
$N$ harmonic-oscillator states. The resulting position and momentum
operators, or quadrature operators in quantum optics, are finite
tridiagonal matrices with the same spectrum. Their eigenvalues are, up to
normalisation, the zeros of the physicists' Hermite polynomial $H_N$.
The smallest and largest positive eigenvalues therefore determine the
smallest and largest
position or momentum scales accessible within this finite-dimensional
model. This spectral interpretation led the authors to examine the product
of the corresponding extreme positive Hermite zeros.

More precisely, let
$$
0<\lambda_{1,N}<\cdots<\lambda_{\lfloor N/2\rfloor,N}
$$
be the positive zeros of $H_N$, and put
$\varpi_N=\lambda_{1,N}\lambda_{\lfloor N/2\rfloor,N}$. Their numerical
studies indicated that
\begin{align}\label{eq:original-hermite-conjecture}
\varpi_N<\varpi_{N+2},\quad N\ge4.
\end{align}
They also reported the asymptotic values $\pi/2$ along the even degrees and
$\pi$ along the odd degrees; see \cite[Section~5]{GazeauJosseMonceau2006}.

The standard Hermite--Laguerre identities are recorded in
\cite[(5.6.1)]{Szego1975}:
\begin{align}\label{eq:hermite-laguerre}
\begin{aligned}
H_{2n}(x)&=(-1)^n2^{2n}n!L_n^{(-1/2)}(x^2),\\[7pt]
H_{2n+1}(x)&=(-1)^n2^{2n+1}n!xL_n^{(1/2)}(x^2).
\end{aligned}
\end{align}
This suggests a more general question concerning Laguerre zeros. For
$\alpha>-1$, use the monic normalisation
\begin{align*}
p_m(x)=(-1)^m m!L_m^{(\alpha)}(x),\quad
0<x_{1,m}^{(\alpha)}<\cdots<x_{m,m}^{(\alpha)},
\quad y_m(\alpha)=x_{1,m}^{(\alpha)}x_{m,m}^{(\alpha)}.
\end{align*}
The positivity and simplicity of these zeros follow from the orthogonality
of the Laguerre polynomials; see \cite[(5.1.1) and
Theorem~3.3.1]{Szego1975}.
We prove the following extension of the conjecture.

\begin{theorem}\label{thm:symmetric-products}
For every $\alpha>-1$, every integer $m\ge2$, and every
$1\le j\le\lfloor m/2\rfloor$,
\begin{align}\label{eq:symmetric-product-main}
x_{j,m-1}^{(\alpha)}x_{m-j,m-1}^{(\alpha)}
<x_{j,m}^{(\alpha)}x_{m+1-j,m}^{(\alpha)}.
\end{align}
Consequently, for each fixed $j\ge1$, the sequence
$$
x_{j,m}^{(\alpha)}x_{m+1-j,m}^{(\alpha)},\quad m\ge2j-1,
$$
is strictly increasing.
\end{theorem}

\begin{figure}[H]
\centering
\begin{tikzpicture}[x=0.85cm,y=1cm,font=\small]
  \draw[gray!55] (0,1.3)--(10,1.3);
  \draw[gray!55] (0,0.3)--(10,0.3);
  \node[anchor=east] at (-0.25,1.3) {degree $5$};
  \node[anchor=east] at (-0.25,0.3) {degree $4$};

  \foreach \x in {0,6.752,10}
    \fill[gray!50] (\x,1.3) circle (1.35pt);
  \foreach \x in {0.522,9.233}
    \fill[gray!50] (\x,0.3) circle (1.35pt);

  \coordinate (A) at (4.339,1.3);
  \coordinate (a) at (4.885,0.3);
  \coordinate (b) at (7.352,0.3);
  \coordinate (B) at (8.504,1.3);
  \fill (A) circle (2.2pt);
  \fill (B) circle (2.2pt);
  \fill[gray!70] (a) circle (2.2pt);
  \fill[gray!70] (b) circle (2.2pt);
  \node[anchor=south east] at (A) {$A$};
  \node[anchor=south west] at (B) {$B$};
  \node[anchor=south east] at (a) {$a$};
  \node[anchor=south west] at (b) {$b$};

  \draw[gray!65,densely dashed] (A)--(4.339,-0.35);
  \draw[gray!65,densely dashed] (a)--(4.885,-0.35);
  \draw[gray!65,densely dashed] (b)--(7.352,-0.35);
  \draw[gray!65,densely dashed] (B)--(8.504,-0.35);
  \draw[<->,line width=0.7pt] (4.339,-0.35)--
    node[midway,above=2pt] {$\delta_L$} (4.885,-0.35);
  \draw[<->,line width=0.7pt] (7.352,-0.35)--
    node[midway,above=2pt] {$\delta_R$} (8.504,-0.35);
  \draw[gray!70,->] (0,-0.75)--(10.2,-0.75)
    node[anchor=west,text=black] {$\log x$};
\end{tikzpicture}
\caption{The case $\alpha=0$, $m=5$, and $j=2$, on the logarithmic scale.
The highlighted zeros are $A=x_{2,5}^{(0)}$, $a=x_{2,4}^{(0)}$,
$b=x_{3,4}^{(0)}$, and $B=x_{4,5}^{(0)}$. Thus
\eqref{eq:symmetric-product-main} is
$\delta_L=\log(a/A)<\log(B/b)=\delta_R$.}
\label{fig:symmetric-pairs}
\end{figure}
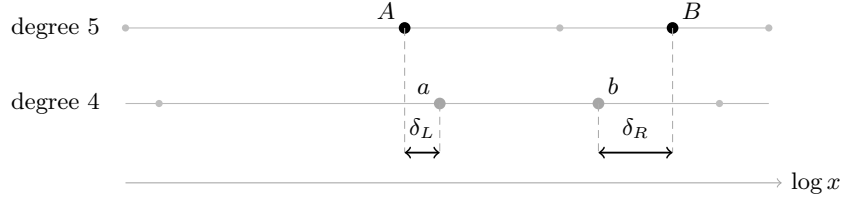

Taking $j=1$ gives $y_{m-1}(\alpha)<y_m(\alpha)$ for $m\ge2$. For the even
and odd degrees, respectively, \eqref{eq:hermite-laguerre} gives
\begin{align*}
\varpi_{2m}^{2}=y_m(-1/2),\quad
\varpi_{2m+1}^{2}=y_m(1/2).
\end{align*}
Thus Theorem~\ref{thm:symmetric-products} proves the conjecture of Gazeau,
Josse-Michaux, and Monceau for both parities and every $N\ge4$. In fact,
it extends their assertion from the extreme positive Hermite zeros to every
symmetrically placed pair of Laguerre zeros and every $\alpha>-1$.

A one-parameter extension of this conjecture was formulated in
\cite[Conjecture~1.1]{Castillo2026}. For $n\ge5$, define
\begin{align*}
\widehat L_{n+1}^{(\alpha)}(x,t)
=\bigl(x-(2n+\alpha+1)\bigr)p_n(x)
-n(n+\alpha)t^2p_{n-1}(x),\quad 0\le t\le1.
\end{align*}
At $t=1$ this is $p_{n+1}$, whereas at $t=0$ its zeros consist of the
zeros of $p_n$ together with $2n+\alpha+1$, counting multiplicities.
As shown below, $\widehat L_{n+1}^{(\alpha)}(\,\cdot\,,t)$ is the
characteristic polynomial of a real symmetric tridiagonal matrix and
therefore has only real zeros. Let
$g_{n,\alpha}(t)$ denote the product of the smallest and largest zeros of
$\widehat L_{n+1}^{(\alpha)}(\,\cdot\,,t)$.

\begin{theorem}\label{thm:deformation}
Let $\alpha>-1$ and let $n\ge5$ be an integer. There exists a unique
$t_*\in(0,1)$ such
that $g_{n,\alpha}$ is strictly increasing on $[0,t_*]$ and strictly
decreasing on $[t_*,1]$. Moreover,
\begin{align}\label{eq:deformation-inequality}
g_{n,\alpha}(t)>g_{n,\alpha}(0),\quad 0<t\le1.
\end{align}
\end{theorem}

The last inequality proves the monic form of Conjecture~1.1 of
\cite{Castillo2026}. The proof will show that the
inequality for $0<t\le1$ reduces to the endpoint $t=1$. We prove the two
theorems in the order stated.

\section{Proof\label{proof}}

\begin{proof}[Proof of Theorem~\ref{thm:symmetric-products}]
The monic Laguerre recurrence is
\begin{align*}
p_{k+1}(x)=\bigl(x-(2k+\alpha+1)\bigr)p_k(x)
-k(k+\alpha)p_{k-1}(x),\quad k\ge1.
\end{align*}
This is the monic form of the recurrence in
\cite[(5.1.10)]{Szego1975}. Expanding a
tridiagonal determinant by its last row gives the same recurrence, with
initial values $1$ and $x-(\alpha+1)$. Hence
$p_k(x)=\det(x\mathcal I_k-\mathcal J_k)$. Thus the zeros of $p_k$ are the
eigenvalues of the Laguerre Jacobi matrix $\mathcal J_k$, whose entries are
\begin{align}\label{eq:Jacobi-entries}
(\mathcal J_k)_{\ell\ell}&=2\ell+\alpha-1,\quad 1\le\ell\le k,\nonumber\\[7pt]
(\mathcal J_k)_{\ell,\ell+1}&=(\mathcal J_k)_{\ell+1,\ell}
=\sqrt{\ell(\ell+\alpha)},\quad 1\le\ell<k.
\end{align}
All remaining entries are zero.
To prove the elementary matrix inequality needed below, we use the
standard monotonicity of ordered eigenvalues in the Loewner order; see
\cite[Corollary~4.3.12]{HornJohnson2013}.

Let $\mathcal J$ be a real symmetric tridiagonal matrix of order $k$. Write
$d_0,\ldots,d_{k-1}$ for its diagonal entries, and let
$\lambda_1(\mathcal J)\le\cdots\le\lambda_k(\mathcal J)$. Then
\begin{align}\label{eq:general-pair-sum}
\lambda_j(\mathcal J)+\lambda_{k+1-j}(\mathcal J)
\le2\max_{0\le\ell<k}d_\ell,\quad 1\le j\le k.
\end{align}
Indeed, let $\mathcal E=\diag(1,-1,1,-1,\ldots)$. Conjugation by
$\mathcal E$ preserves the diagonal and changes the sign of every
off-diagonal entry. For real symmetric matrices $\mathcal S$ and $\mathcal T$
of the same order, we write $\mathcal S\preceq\mathcal T$ when
$\mathcal T-\mathcal S$ is positive semidefinite. Put
$d_*=\max_{0\le\ell<k}d_\ell$, and let $\mathcal I_k$ be the identity matrix of
order $k$. Then
\begin{align*}
\mathcal J+\mathcal E\mathcal J\mathcal E
&=2\diag(d_0,\ldots,d_{k-1}),\\[7pt]
\mathcal J&\preceq2d_*\mathcal I_k-\mathcal E\mathcal J\mathcal E.
\end{align*}
Indeed, the matrix on the right of the second inequality minus $\mathcal J$ is
$2\diag(d_*-d_0,\ldots,d_*-d_{k-1})$, which is positive semidefinite.
The $j$th eigenvalue of the matrix on the right, in increasing order, is
$2d_*-\lambda_{k+1-j}(\mathcal J)$. Monotonicity of ordered eigenvalues under the
Loewner order proves \eqref{eq:general-pair-sum}.

Equations \eqref{eq:Jacobi-entries} and
\eqref{eq:general-pair-sum} give
\begin{align}\label{eq:Laguerre-pair-sum}
x_{j,k}^{(\alpha)}+x_{k+1-j,k}^{(\alpha)}
\le2(\alpha+2k-1),\quad 1\le j\le k.
\end{align}
Fix $m\ge2$ and $1\le j\le\lfloor m/2\rfloor$, and abbreviate
\begin{align*}
A=x_{j,m}^{(\alpha)},\quad
B=x_{m+1-j,m}^{(\alpha)},\quad
a=x_{j,m-1}^{(\alpha)},\quad
b=x_{m-j,m-1}^{(\alpha)}.
\end{align*}
When $m=2j$, the degree $m-1$ is odd and $a=b$ is its central zero.
Strict interlacing \cite[Theorem~3.3.2]{Szego1975} gives
\begin{align}\label{eq:interlacing-AabB}
A<a\le b<B.
\end{align}
Put $c=2m+\alpha$ and $\beta=m(m+\alpha)>0$.
Applied in degrees $m$ and $m-1$, respectively,
\eqref{eq:Laguerre-pair-sum} becomes
\begin{align}\label{eq:two-pair-sum-bounds}
A+B&\le2(c-1),\nonumber\\[7pt]
a+b&\le2(c-3).
\end{align}

Here and below, primes on polynomials and rational functions denote
derivatives with respect to $x$. The Laguerre derivative identity
\cite[(5.1.14)]{Szego1975} gives
$xp_m'=mp_m+\beta p_{m-1}$. Define
\begin{align*}
g(x)=x\frac{p_m'(x)}{p_m(x)}-m
=\beta\frac{p_{m-1}(x)}{p_m(x)}
=\sum_{\ell=1}^{m}\frac{x_{\ell,m}^{(\alpha)}}
{x-x_{\ell,m}^{(\alpha)}}.
\end{align*}
The Laguerre differential equation \cite[(5.1.2)]{Szego1975} then gives
\begin{align}\label{eq:g-equation}
xg'=-g^2+(x-c)g-\beta.
\end{align}
Since all the residues are positive,
\begin{align*}
g'(x)=-\sum_{\ell=1}^{m}\frac{x_{\ell,m}^{(\alpha)}}
{(x-x_{\ell,m}^{(\alpha)})^2}<0.
\end{align*}
Strict interlacing therefore shows that $g$ is positive on
$(A,a)$ and negative on $(b,B)$, with
\begin{align*}
g(A^+)&=+\infty,\quad g(a)=0,\\[7pt]
g(b)&=0,\quad g(B^-)=-\infty.
\end{align*}
Set $\delta_L=\log(a/A)$ and $\delta_R=\log(B/b)$. Suppose, to the
contrary, that $ab\ge AB$, or equivalently $\delta_L\ge\delta_R$. For
$0<s<\delta_L$ and $0<s<\delta_R$, respectively, use the principal
branch of the arctangent, with values in $(-\pi/2,\pi/2)$, to define
\begin{align*}
\theta_L(s)&=\arctan\frac{\sqrt\beta}{g(Ae^s)},\\[7pt]
\theta_R(s)&=\arctan\frac{\sqrt\beta}{-g(Be^{-s})}.
\end{align*}
Extend both functions continuously to their endpoints. Then
\begin{align}\label{eq:function-endpoints}
\theta_L(0)=\theta_R(0)=0,\quad
\theta_L(\delta_L)=\theta_R(\delta_R)=\frac{\pi}{2}.
\end{align}
Under the assumption $\delta_L\ge\delta_R$, both functions are defined on
the common interval $[0,\delta_R]$.
Equation \eqref{eq:g-equation} gives
\begin{align}\label{eq:function-equations}
\theta_L'&=\sqrt\beta-\frac{Ae^s-c}{2}\sin(2\theta_L),\nonumber\\[7pt]
\theta_R'&=\sqrt\beta-\frac{c-Be^{-s}}{2}\sin(2\theta_R).
\end{align}
The primes in these two equations denote derivatives with respect to $s$.

Put $D(s)=Ae^s+Be^{-s}-2c$. This function is
strictly convex, and \eqref{eq:two-pair-sum-bounds} gives
\begin{align*}
D(0)&=A+B-2c\le-2,\\[7pt]
D(\delta_R)&=\frac{AB}{b}+b-2c\le a+b-2c\le-6.
\end{align*}
Here we used $Be^{-\delta_R}=b$ and
$Ae^{\delta_R}=AB/b\le a$. A convex function lies below the chord joining
two of its points, and hence $D(s)<0$ for $0\le s\le\delta_R$.
Let $w=\theta_L-\theta_R$, and define the continuous function
\begin{align*}
r(s)=-(Ae^s-c)\int_0^1
\cos\bigl(2\theta_R(s)+2\tau w(s)\bigr)\,\mathrm d\tau.
\end{align*}
Subtracting the equations in \eqref{eq:function-equations} gives
\begin{align*}
w'=r(s)w+h(s),\quad
h(s)=-\frac{D(s)}2\sin(2\theta_R(s)).
\end{align*}
For $0<s<\delta_R$, we have
$0<\theta_R(s)<\pi/2$ and $D(s)<0$, and hence $h(s)>0$. Both $r$ and $h$
extend continuously to the closed interval, and
$w(0)=0$ by \eqref{eq:function-endpoints}. Since $r$ is bounded and
$w(\varepsilon)\to0$, applying variation of constants on
$[\varepsilon,s]$ and then letting $\varepsilon\downarrow0$ gives
\begin{align*}
w(s)=\int_0^s\exp\left(\int_\tau^s r(\xi)\,\mathrm d\xi\right)
h(\tau)\,\mathrm d\tau>0,\quad 0<s\le\delta_R.
\end{align*}
This
is impossible because \eqref{eq:function-endpoints} and
$\delta_L\ge\delta_R$ give
$\theta_R(\delta_R)=\pi/2$ and $\theta_L(\delta_R)\le\pi/2$.
Therefore $ab<AB$, which is \eqref{eq:symmetric-product-main}. Applying
this inequality for $m=2j,2j+1,\ldots$ proves the final assertion of the
theorem.
\end{proof}

We now prove the one-parameter consequence.

\begin{proof}[Proof of Theorem~\ref{thm:deformation}]
Fix $\alpha>-1$ and $n\ge5$. Put $a=2n+\alpha+1$,
$b=\sqrt{n(n+\alpha)}$, which is positive, and $s=t^2$, and set
\begin{align}\label{eq:F-def}
F(x,s)=(x-a)p_n(x)-b^2s\,p_{n-1}(x).
\end{align}
Let $u(s)<v(s)$ be its smallest and largest zeros, and put
$G(s)=u(s)v(s)$, $U(s)=a-u(s)$, and $V(s)=v(s)-a$.
The zeros of $F$ are the eigenvalues of
\begin{align}\label{eq:deformed-Jacobi}
\mathcal J_{n+1}(t)=
\begin{pmatrix}
\mathcal J_n&tb\,e\\[7pt]
tb\,e^{\mathrm T}&a
\end{pmatrix},
\end{align}
where $e$ is the last coordinate vector in $\mathbb R^n$. Expansion along
the last row gives
$\det(x\mathcal I_{n+1}-\mathcal J_{n+1}(t))=F(x,t^2)$. The matrix in
\eqref{eq:deformed-Jacobi} is positive definite. Indeed, for $t>0$, let
$\mathcal D_t=\diag(\mathcal I_n,t)$. Then
\begin{align*}
\mathcal J_{n+1}(t)=\mathcal D_t\mathcal J_{n+1}(1)\mathcal D_t
+a(1-t^2)e_{n+1}e_{n+1}^{\mathrm T}.
\end{align*}
Here $e_{n+1}$ denotes the last coordinate vector in $\mathbb R^{n+1}$.
Since $\mathcal J_{n+1}(1)$ is the Laguerre Jacobi matrix and $\mathcal D_t$
is nonsingular,
the first term is positive definite and the second is positive semidefinite.
At $t=0$ the matrix is $\diag(\mathcal J_n,a)$, which is also positive definite. In
particular, $u$, $v$, and $G$ are positive.

We shall also need the position of $a$. We use the Rayleigh--Ritz principle
and Cauchy's interlacing theorem in the forms given in
\cite[Theorems~4.2.2 and~4.3.28]{HornJohnson2013}. The Rayleigh quotient at the
first coordinate vector gives $x_{1,n}^{(\alpha)}\le\alpha+1$. Equality
would make the first coordinate vector an eigenvector of $\mathcal J_n$. This is
impossible because the first off-diagonal entry is positive. Hence
$x_{1,n}^{(\alpha)}<\alpha+1<a$. The largest eigenvalue
of the final $2\times2$ principal block of $\mathcal J_n$ is
$a-3+\sqrt{1+(n-1)(n+\alpha-1)}$. This number is greater than $a$, since
$(n-1)(n+\alpha-1)>(n-1)(n-2)\ge12>8$. Interlacing therefore gives
$x_{1,n}^{(\alpha)}<a<x_{n,n}^{(\alpha)}$.
The same interlacing, applied to \eqref{eq:deformed-Jacobi}, shows that
$u<a<v$ for $0\le s\le1$.
For $s>0$ the matrix is an irreducible Jacobi matrix, so all its
eigenvalues are simple. Indeed, the three-term eigenvector recurrence
determines an eigenvector from its first coordinate, which cannot vanish.
At $s=0$, the strict position of $a$ shows that the two extreme
eigenvalues are also simple. The ordered eigenvalues are continuous in
the matrix entries, and the implicit-function theorem at each simple
extreme root shows that $u$ and $v$ are continuous on $[0,1]$ and
continuously differentiable on $(0,1)$.

We first derive the velocity of a zero $x=x(s)$ of $F$. The recurrence and
derivative identity \cite[(5.1.10) and (5.1.14)]{Szego1975} give
\begin{align*}
xp_n'&=np_n+b^2p_{n-1},\\[7pt]
xp_{n-1}'&=(x-n-\alpha)p_{n-1}-p_n.
\end{align*}
For $s>0$, a zero $x$ of $F$ is positive and $p_n(x)\ne0$, since otherwise
$p_n$ and $p_{n-1}$ would have a common zero. Using
$b^2s p_{n-1}(x)=(x-a)p_n(x)$ and writing
$F_x=\partial F/\partial x$, we first obtain
\begin{align*}
xF_x={}&\bigl(x+n(x-a)+b^2s\bigr)p_n+b^2\bigl((x-a)-s(x-n-\alpha)\bigr)p_{n-1}.
\end{align*}
Substitution of the root relation now gives
\begin{align*}
F_x(x,s)=\frac{p_n(x)}{xs}
\bigl((1-s)(x-a)^2+as+b^2s^2\bigr).
\end{align*}
Here $s>0$. Implicit differentiation of \eqref{eq:F-def} therefore yields
\begin{align}\label{eq:zero-velocity}
x_s=\frac{x(x-a)}{(1-s)(x-a)^2+as+b^2s^2}.
\end{align}
Here $x_s=\mathrm dx/\mathrm ds$; all subsequent subscripts $s$ have the
same meaning.
The last formula extends continuously from the right to the extreme
branches at $s=0$;
since neither branch equals $a$ there, its limit is $x/(x-a)$.

The strict comparison between $U$ and $V$ follows from the
Perron--Frobenius theorem for primitive matrices; we use the formulation in
\cite[Section~8.5]{HornJohnson2013}.

Put $\mathcal E=\diag(1,-1,\ldots,(-1)^n)$, of order $n+1$. For $s>0$,
the matrix
$$
\mathcal E\bigl(a\mathcal I_{n+1}-\mathcal J_{n+1}(\sqrt{s})\bigr)\mathcal E
$$
is entrywise non-negative and primitive. Here primitive means
that some power is entrywise positive; irreducibility together with a
positive diagonal entry gives this property. Its Perron root, or spectral
radius, is $U$, whereas $-V$ is another eigenvalue. Thus $U>|V|$, and in
particular $U>V$.
At $s=0$, the position of $a$ proved above shows that $u$ and $v$ are the
extreme eigenvalues of $\mathcal J_n$. The leading block of the displayed Perron
matrix is primitive, with Perron root $U$ and another eigenvalue $-V$, so
$U>V$ also at this endpoint.

Introduce
\begin{align*}
D_U&=(1-s)U^2+as+b^2s^2,\quad
D_V=(1-s)V^2+as+b^2s^2,\\[7pt]
\Phi(s)&=(1-s)UV-as-b^2s^2.
\end{align*}
The quantities $D_U$ and $D_V$ are positive on $[0,1]$: at $s=0$ they
are $U^2$ and $V^2$, while for $s>0$ the term $as$ is positive.
Applying \eqref{eq:zero-velocity} to $u$ and $v$ gives
\begin{align}\label{eq:product-velocity}
\frac{G_s}{G}
=-\frac{U}{D_U}+\frac{V}{D_V}
=\frac{(U-V)\Phi}{D_U D_V}.
\end{align}
Let $s_0\in(0,1)$ be any zero of $\Phi$. Then
\begin{align*}
U(s_0)V(s_0)=\frac{as_0+b^2s_0^2}{1-s_0}>0.
\end{align*}
Substitution in \eqref{eq:zero-velocity} gives
\begin{align*}
U_s(s_0)&=\frac{a-U(s_0)}{(1-s_0)(U(s_0)+V(s_0))},\\[7pt]
V_s(s_0)&=\frac{a+V(s_0)}{(1-s_0)(U(s_0)+V(s_0))}.
\end{align*}
Differentiating $\Phi$ along the two extreme-zero branches and evaluating
at $s_0$, we use
\begin{align*}
(1-s_0)\bigl(U_s(s_0)V(s_0)+U(s_0)V_s(s_0)\bigr)=a.
\end{align*}
It follows that
\begin{align}\label{eq:phi-downcrossing}
\Phi'(s_0)
=-U(s_0)V(s_0)-2b^2s_0<0.
\end{align}
Here $\Phi'$ is the ordinary derivative with respect to $s$ after
$U=U(s)$ and $V=V(s)$ have been substituted; $s_0$ is simply a zero at
which that derivative is evaluated.

From the position of $a$ obtained above,
\begin{align*}
\Phi(0)&=\bigl(a-x_{1,n}^{(\alpha)}\bigr)
\bigl(x_{n,n}^{(\alpha)}-a\bigr)>0,\\[7pt]
\Phi(1)&=-a-b^2<0.
\end{align*}
The function $\Phi$ is continuous on $[0,1]$ and continuously
differentiable on $(0,1)$. Every interior zero is simple and is a strict
downward crossing by \eqref{eq:phi-downcrossing}; hence its zeros are
isolated. Its zero set is closed, and the endpoint signs give zero-free
neighbourhoods of $0$ and $1$. If it were infinite, compactness would give an
interior accumulation point, contradicting isolation. Thus there are only
finitely many zeros. If there were two consecutive ones, $\Phi$ would be
negative immediately after the first and positive immediately before the
second, which would force an intervening zero that is not a downward crossing.
Thus there is exactly one zero, say $s_*$. Equation
\eqref{eq:product-velocity} and $U>V$ prove that $G$ is strictly increasing
on $[0,s_*]$ and strictly decreasing on $[s_*,1]$.

Finally, since $a$ lies between the two extreme zeros of $p_n$,
\begin{align*}
F(x,0)&=(x-a)p_n(x),\quad F(x,1)=p_{n+1}(x),\\[7pt]
G(0)&=y_n(\alpha),\quad \ \ \ \ \ \ \ \ \ \ \ \ G(1)=y_{n+1}(\alpha).
\end{align*}
\begin{samepage}
The strict increase of $G$ on $[0,s_*]$ gives $G(s)>G(0)$ on
$(0,s_*]$, while its strict decrease on $[s_*,1]$ gives
$G(s)>G(1)$ on $[s_*,1)$. Consequently,
$G(s)>G(0)$ for every $0<s\le1$ exactly when
$y_{n+1}(\alpha)>y_n(\alpha)$. Theorem~\ref{thm:symmetric-products}, with
$j=1$ and $m=n+1$, gives this endpoint inequality. Since
$t\mapsto t^2$ is strictly increasing on $[0,1]$, the monotonicity of $G$
transfers to $g_{n,\alpha}$. Returning to $s=t^2$ proves
\eqref{eq:deformation-inequality}, and the unique maximum occurs at
$t_*=\sqrt{s_*}$.
\end{samepage}
\end{proof}

\begin{remark}
For completeness, we record three minor corrections to
\cite{Castillo2026}. Up to a non-zero constant factor, the
ordinary-Laguerre form in Proposition~1.1 has
$+(n+\alpha)t^2L_{n-1}^{(\alpha)}$ as its last term. In Theorem~2.1,
$P_{n+1}$ denotes the characteristic polynomial of the undeformed matrix
$\mathcal J_{n+1}$, and the sign formula involving the auxiliary deformation
function $f$ is understood away from $f=\pm1$. Finally, the
diagonal-dominance comment in
the last paragraph applies directly to $\alpha>0$. These local points do
not change the intended numbered results of that paper and none is used
in the proofs above.
\end{remark}

\section*{Acknowledgements}

This work was supported by the Portuguese Foundation for Science and
Technology through the Centre for Mathematics of the University of Coimbra,
project UID/00324/2025, and through project
2022.00143.\allowbreak CEECIND/\allowbreak CP1714/\allowbreak CT0002.

\bibliographystyle{plain}
\bibliography{bib}

\begin{thebibliography}{1}

\bibitem{Castillo2026}
K.~Castillo.
\newblock On the product of the extreme zeros of {Laguerre} polynomials.
\newblock {\em Results Math.}, 81(1), 2026.
\newblock Art.~9, 14 pp.

\bibitem{GazeauJosseMonceau2006}
J.-P. Gazeau, F.-X. Josse-Michaux, and P.~Monceau.
\newblock Finite dimensional quantizations of the $(q,p)$ plane: New space and
  momentum (or quadratures) inequalities.
\newblock {\em Int. J. Mod. Phys. B}, 20(11--13):1778--1791, 2006.

\bibitem{HornJohnson2013}
R.~A. Horn and C.~R. Johnson.
\newblock {\em Matrix Analysis}.
\newblock Cambridge University Press, New York, 2nd edition, 2013.
\newblock Corrected reprint, 2018.

\bibitem{Szego1975}
G.~Szeg\H{o}.
\newblock {\em Orthogonal Polynomials}, volume~23 of {\em American Mathematical
  Society Colloquium Publications}.
\newblock American Mathematical Society, Providence, RI, 4th edition, 1975.

\end{thebibliography}

\end{document}